\documentclass[10pt]{article}
\usepackage{graphicx}
\usepackage{tikz}
\usepackage[margin=1in]{geometry}
\usepackage{amsmath, amssymb, amsfonts, amsthm}
\newtheorem{lem}{\noindent {\bf Lemma}}[section]
\newtheorem{prop}{\noindent {\bf Proposition}}[section]

\newtheorem{thm}{\noindent {\bf Theorem}}[section]

\newcounter{remark}
\newenvironment{remark}{\smallskip\noindent {\bf Remark \arabic{section}.\arabic{remark}.}}
{\addtocounter{remark}{1}\par}
\newcounter{example}
\catcode`@=11 
\date{}
\newcounter{defi}
\newenvironment{defi}{
\smallskip \noindent
{\bf
  Definition \arabic{section}.\arabic{defi}.
}}{\addtocounter{defi}{1}\par}
\newcommand{\ZZ}{\mathbb Z}
\newcommand{\NN}{\mathbb N}
\newcommand{\RR}{\mathbb R}

\title{\bf A metric space with its transfinite asymptotic dimension~$\omega+1$}
\author{\large  Yan Wu$^\ast$\qquad Jingming Zhu$^{\ast\ast}$
\footnote{
College of Mathematics Physics and Information Engineering, Jiaxing University, Jiaxing , 314001, P.R.China.
$^\ast$ E-mail: yanwu@mail.zjxu.edu.cn $^\ast\ast$ E-mail: jingmingzhu@mail.zjxu.edu.cn }}
\date{}

\begin{document}
\maketitle
\begin{center}
\begin{minipage}{0.9\textwidth}
\noindent{\bf Abstract.}
We construct a metric space whose transfinite asymptotic dimension and complementary-finite asymptotic dimension are both $\omega+1$,
where $\omega$ is the smallest infinite ordinal number. Therefore, we prove that the omega conjecture is not true.


\end{minipage}
\end{center}
\footnote{
This research was supported by
the National Natural Science Foundation of China under Grant (No.11871342,11301224,11326104,11401256,11501249)

}
\begin{section}{Introduction}\

In coarse geometry, asymptotic dimension of a metric space is an important concept which was defined by Gromov for studying asymptotic invariants of discrete groups \cite{Gromov}. This dimension can be considered as an asymptotic analogue of the Lebesgue covering dimension.  T. Radul defined the trasfinite asymptotic dimension (trasdim) which can be viewed as a transfinite extension for asymptotic dimension and gave examples of metric spaces with trasdim$=\omega$ and with trasdim$=\infty$, where $\omega$ is the smallest infinite ordinal number (see \cite{Radul2010}). But whether there is a metric space $X$ with $\omega<$trasdim$(X)<\infty$ is still unknown so far. M. Satkiewicz stated "omega conjecture"(see \cite{satkiewicz}). That is,  $\omega\leq$trasdim$(X)<\infty$ implies trasdim$(X)=\omega$. In this paper, we prove that the "omega conjecture" is not true by constructing a metric space $X_{\omega+1}$ with trasdim$(X_{\omega+1})=\omega+1$.

In the paper \cite{yanzhu2018}, we introduced another approach to classify the metric spaces with infinite asymptotic dimension, which is called complementary-finite asymptotic dimension (coasdim), and gave metric spaces $X_k$ with coasdim$(X_k)=\omega+k$ and trasdim$(X_k)=\omega$, where $k\in \NN$. The metric space $X_{\omega+1}$ constructed in this paper is the first metric space $Y$ with trasdim$(Y)>\omega$ and trasdim$(Y)=$coasdim$(Y)$.

The paper is organized as follows: In Section 2, we recall some definitions and properties of transfinite asymptotic dimension. In
Section 3, we introduce a concrete metric space $X_{\omega+1}$ and prove that transfinite asymptotic dimension and complementary-finite asymptotic dimension of $X_{\omega+1}$ are both $\omega+1$.
\end{section}

\begin{section}{Preliminaries}\

Our terminology concerning the asymptotic dimension follows from \cite{Bell2011} and for undefined terminology we refer to \cite{Radul2010}. Let~$(X, d)$ be a metric space and $U,V\subseteq X$, let
\[
\text{diam}~ U=\text{sup}\{d(x,y)| x,y\in U\}
\text{   and   }
d(U,V)=\text{inf}\{d(x,y)| x\in U,y\in V\}.
\]
Let $R>0$ and $\mathcal{U}$ be a family of subsets of $X$. $\mathcal{U}$ is said to be \emph{$R$-bounded} if
\[
\text{diam}~\mathcal{U}\stackrel{\bigtriangleup}{=}\text{sup}\{\text{diam}~ U~|~ U\in \mathcal{U}\}\leq R.
\]
In this case, $\mathcal{U}$ is said to be \emph{uniformly bounded}.
Let $r>0$, a family $\mathcal{U}$ is said to be\emph{ $r$-disjoint} if
\[
d(U,V)\geq r~~~~~\text{for every}~ U,V\in \mathcal{U}\text{~with~}U\neq V.
\]
In this paper, we denote
$\bigcup\{U~|~U\in\mathcal{U}\}$ by $\bigcup\mathcal{U}$ and denote $\{U~|~U\in\mathcal{U}_{1}\text{~or~}~U\in\mathcal{U}_{2}\}$
by $\mathcal{U}_{1}\cup\mathcal{U}_{2}$.

\begin{defi}
A metric space $X$ is said to have \emph{finite asymptotic dimension} if
there is an $n\in\NN$, such that for every $r>0$,
there exists a sequence of uniformly bounded families
$\{\mathcal{U}_{i}\}_{i=0}^{n}$ of subsets of $X$
such that the family
$\bigcup_{i=0}^{n}\mathcal{U}_{i}$ covers $X$ and each $\mathcal{U}_{i}$
is $r$-disjoint for $i=0,1,\cdots,n$. In this case, we say that
the asymptotic dimension of $X$ less than or equal to $n$, which is denoted by
 asdim$X\leq n$.
We say that asdim$X= n$ if  asdim$X\leq n$ and asdim$X\leq n-1$ is not true.

\end{defi}

%
T. Radul generalized asymptotic dimension of a metric space $X$ to transfinite asymptotic dimension which is denoted by trasdim$(X)$ (see \cite{Radul2010}).
Let $Fin\mathbb{N}$ denote the collection of all finite, nonempty
subsets of $\mathbb{N}$ and let $ M \subseteq Fin\mathbb{N}$. For $\sigma\in \{\varnothing\}\bigcup Fin\mathbb{N}$, let
$$M^{\sigma} = \{\tau\in Fin\mathbb{N} ~|~ \tau \cup \sigma \in M \text{ and } \tau \cap \sigma = \varnothing\}.$$

Let $M^a$ abbreviate $M^{\{a\}}$ for $a \in \NN$. Define the ordinal number Ord$M$ inductively as follows:
\begin{eqnarray*}
\text{Ord}M = 0 &\Leftrightarrow& M = \varnothing,\\
\text{Ord}M \leq \alpha &\Leftrightarrow& \forall~ a\in \mathbb{N}, ~\text{Ord}M^a < \alpha,\\
\text{Ord}M = \alpha &\Leftrightarrow& \text{Ord}M \leq \alpha \text{ and } \text{Ord}M < \alpha \text{ is not true},\\
\text{Ord}M = \infty &\Leftrightarrow& \text{Ord}M \leq\alpha \text{ is not true for every ordinal number } \alpha.
\end{eqnarray*}

%

Given a metric space $(X, d)$, define the following collection:
\[
\begin{split}
A(X, d) = \{\sigma \in Fin\mathbb{N}~ |~&\text{ there are no uniformly bounded families } \mathcal{U}_i  \text{ for } i \in \sigma
 \\& \text{ such that each } \mathcal{U}_i
\text{ is } i\text{-disjoint and }\bigcup_{i\in\sigma}\mathcal{U}_i \text{~covers~} X\}.
\end{split}\]

The \emph{transfinite asymptotic dimension} of $X$ is defined as trasdim$X$=Ord$A(X, d)$.
\end{section}
\begin{section}{Main result}\

%
Let $A$ be a subset of a metric space $X$ and $\epsilon>0$, we denote $\{x\in X|d(x,A)<\epsilon\}$ by $N_{\epsilon}(A)$. For the unit cube $I^n$, facets is the subset with the $i$-th coordinate $x_i=0$ or $x_i=1$ for some integer $i$ $(1\leq i\leq n)$.

To prove the main result, we will use another version of Lebesgue theorem:
\begin{thm}\rm(see \cite{Karasev}, Theorem 4.3)
\label{thmcube}
Let the unit cube $I^n$ be covered by a family of closed sets $\{X_i\}_{1\leq i\leq n}$. Then some connected component of $X_i$ intersects both the corresponding opposite facets $A_i$ and $B_i$.
\end{thm}

\begin{lem}\rm(see \cite{yanzhu2018}, Proposition 2.1)
\label{lem:trasdim}
Given a metric space $X$,  let $l\in \mathbb{N}\cup \{0\}$, then the following conditions are equivalent:
\begin{itemize}
\item\rm(1) trasdim$(X) \leq \omega+l$;

\item\rm(2) For every $k\in\NN$, there exists $m=m(k)\in\NN$ such that for every~$n\in\NN$, there are uniformly bounded families $\mathcal{U}_{-l},\mathcal{U}_{-l+1},\cdots,\mathcal{U}_{m}$ satisfying $\mathcal{U}_i$ is $k$-disjoint for $i=-l,\cdots, 0$, $\mathcal{U}_j$ is $n$-disjoint for $j=1,2,\cdots, m$ and
$\bigcup_{i=-l}^{m}\mathcal{U}_i$ covers $X$.
\end{itemize}
\end{lem}

Let $X_1=\mathbb{R}$, $X_2=(\mathbb{R}\times 2\mathbb{Z})\bigcup (2\mathbb{Z}\times\mathbb{R} )$,
$X_3=(\mathbb{R}\times 3\mathbb{Z}\times 3\mathbb{Z})\bigcup (3\mathbb{Z}\times\mathbb{R}\times 3\mathbb{Z} )\bigcup (3\mathbb{Z}\times 3\mathbb{Z}\times \mathbb{R} )$,...
For $a=(a_1,a_2,...,a_l)\in X_{l}$ and $b=(b_1,b_2,...,b_k)\in X_{k}$ with $l\leq k$, let $a'=(a_1,...,a_l,0,...,0)$,$b'=(b_1,...,b_k,0,...,0)\in \bigoplus \mathbb{R}$ and put $c=0$ if $l=k$ and $c=l+(l+1)+...+(k-1)$ if $l<k$. Define a metric $d$ on $\bigsqcup X_n$ by
\[d(a,b)=\text{max}\{d_{\text{max}}(a',b'),c\},\]
where $d_{\text{max}}$ is the maximum metric in $\bigoplus\mathbb{R}.$
Let $X_{\omega+1}=(\bigsqcup X_n,d)$.

\begin{prop}
\label{prop1}
trasdim$(X_{\omega+1})\leq\omega$ is not true.
\end{prop}

\begin{proof}
Suppose that trasdim$(X_{\omega+1})\leq \omega$. Then by Lemma \ref{lem:trasdim}, for every $k\in\NN$, there exists $m=m(k)\in\NN$ such that for every~$n\in\NN$, there exist $B=B(k,n)>0$ and $B$-bounded families $\mathcal{V}_{0},\mathcal{V}_{1},\cdots,\mathcal{V}_{m}$ such that $\mathcal{V}_0$ is $k$-disjoint, $\mathcal{V}_j$ is $(n+4m+4)$-disjoint for $j=1,2,\cdots, m$ and
$\bigcup_{i=0}^{m}\mathcal{V}_i$ covers $X$.
Let $\mathcal{U}_i=\{U\cap X_{m+1}\cap [0,B+n]^{m+1}|U\in \mathcal{V}_i \}$ for $i=0,1,2,\cdots, m$. \

Without loss of generality, we assume that $n\in\NN$, $n>4m+4$ and $\frac{B+n}{m+1}=p\in \mathbb{N}$. Then $[0,B+n]^{m+1}$ contains $(m+1)$-dimensional closed cubes $V_1,V_2,...,V_{p^{m+1}}$, where $V_i=\prod_{j=1}^{m+1}[(m+1)n_j^{(i)},(m+1)(n_j^{(i)}+1)]$, whose edges are contained in $X_{m+1}$ for $i=1,2,\cdots, p^{m+1}$.

Let $\mathcal{V}=\{V_j\}_{1\leq j\leq p^{m+1}}$ and $A_i$, $B_i$ be the pair of two opposite facets of $[0,B+n]^{m+1}$ for $i=1,2,\cdots,m+1$.

Let $\widetilde{\mathcal{U}_i}=\{V\in \mathcal{V}|V\cap U\neq\emptyset \text{ for some } U\in \mathcal{U}_i\}$ for $1\leq i\leq m$.

Then, by the definition, $\widetilde{\mathcal{U}_i}$ is $(n+4m+4)$-disjoint and $(B+2m+2)$-bounded subset family in $[0,B+n]^{m+1}$ for $1\leq i\leq m$.

Let
\[
\widetilde{U_i}=N_{m+1}(\bigcup \mathcal{U}_i)\text{ and }\widetilde{U_0}=\overline{[0,B+n]^{m+1}\setminus \bigcup_{1\leq i\leq m}\widetilde{U_i}}
\]

Then $\{\widetilde{U}_i\}_{0\leq i\leq m}$ is a closed cover of $[0,B+n]^{m+1}$ with covering multiplicity at most $m+1$. By the Theorem \ref{thmcube}, $\exists$ at least one $\widetilde{U}_i$ which intersects both the corresponding opposite facets $A_i$ and $B_i$. But every connected component in $\widetilde{U}_i$ for $1\leq i\leq m$ is $(B+2m+2)$-bounded, so $\widetilde{U_0}\neq \emptyset$ and there is a path $\alpha:I\to \widetilde{U}_0$ such that $\alpha(0)$ and $\alpha(1)$ belong to opposite facets $A_i$ and $B_i$ respectively.

Let $\mathcal{W}=\{V\in\mathcal{V}|V\cap \alpha(I)\neq \emptyset\}$. Then $\bigcup \mathcal{W}$ is a connected subset in $[0,B+n]^{m+1}$ which intersects both the corresponding opposite facets $A_i$ and $B_i$. But $\alpha(I)\subset \widetilde{U}_0$ implies that any $V\in \mathcal{V}$ can not intersect both $\alpha(I)$ and $\bigcup{\mathcal{U}_i}$ for any $1\leq i\leq m$. Then $(\bigcup \mathcal{W})\cap (\bigcup\mathcal{U}_i)=\emptyset$ for $1\leq i\leq m$. So $\exists$ a path $\beta:I\to [0,B+n]^{m+1}$ such that $\beta(I)$ is a subset of the 1-dimensional skeleton of the $V$'s in $\mathcal{W}$(and hence a subset of $X_{m+1}$) and intersects both the corresponding opposite facets $A_i$ and $B_i$. Since $\bigcup \mathcal{W}\cap (\bigcup\mathcal{U}_i)=\emptyset$ for $1\leq i\leq m$ and $\beta(I)\subset \bigcup \mathcal{W}$, $\beta(I)\cap (\bigcup\mathcal{U}_i)=\emptyset$ for $1\leq i\leq m$. So $\beta(I)\subset (\bigcup\mathcal{U}_0)\cap X_{m+1}(=\bigcup\mathcal{V}_0)$ which is a contradiction with the fact that $\mathcal{V}_0$ is $B$-bounded and $k$-disjoint.

\end{proof}

\begin{defi}
Every ordinal number $\gamma$ can be represented as $\gamma=\lambda(\gamma)+n(\gamma)$, where $\lambda(\gamma)$ is the limit ordinal or $0$ and $n(\gamma)\in \mathbb{N}\cup \{0\}$. Let $X$ be a metric space, we define complementary-finite asymptotic dimension coasdim$(X)$ inductively as follows:
\begin{itemize}
\item\text{coasdim}$(X)=-1$ iff $X=\emptyset$,

\item \text{coasdim}$(X)\leq \gamma$ iff for every $r>0$ there exist $r$-disjoint uniformly bounded families $\mathcal{U}_0,...,\mathcal{U}_{n(\gamma)}$ of subsets of $X$ such that \text{coasdim}$(X\setminus \bigcup(\bigcup_{i=0}^{n(\gamma)}\mathcal{U}_i))<\lambda(\gamma)$,

\item \text{coasdim}$(X)=\gamma$ iff \text{coasdim}$(X)\leq \gamma$ and for every $\beta<\gamma$,
\text{coasdim}$(X)\leq\beta$ is not true.

\item \text{coasdim}$(X)=\infty$ iff for every ordinal number $\gamma$, \text{coasdim}$(X)\leq \gamma$ is not true.
\end{itemize}
\end{defi}

\begin{remark}
\begin{itemize}
\item
It is easy to see that for every $n\in\mathbb{N}\cup \{0\}$, \text{coasdim}$(X)\leq n$ if and only if \text{asdim}$(X)\leq n$.
\item
To simplify, we will abuse the notation a little bit. In the following, asdim$(X)<\omega$ means that asdim$(X)\leq n$ for some $n\in\mathbb{N}$.
\end{itemize}

\end{remark}

\begin{lem}\rm(see~\cite{yanzhu2018}, Theorem 3.2)
\label{lem:coasdim}
Let $X$ be a metric space, if coasdim$(X)\leq \omega+k$ for some $k\in\mathbb{N}$, then trasdim$(X)\leq\omega+k$. Especially, coasdim$(X)=\omega$ implies trasdim$(X)=\omega$.
\end{lem}

\begin{prop}
\label{prop2}
trasdim$(X_{\omega+1})\leq\omega+1$.
\end{prop}
\begin{proof}
By the Lemma \ref{lem:coasdim}, it suffices to show that coasdim$(X_{\omega+1})\leq \omega+1$. i.e., for every $k>0$ there exist $k$-disjoint uniformly bounded families $\mathcal{U}_0,\mathcal{U}_{1}$ of subsets of $X_{\omega+1}$ such that \text{coasdim}$(X_{\omega+1}\setminus \bigcup(\mathcal{U}_0\cup\mathcal{U}_1))<\omega$.\
Let $n\in\NN$ and $n\geq 6k$, let
$\mathcal{V}_0^{(n)}=\{[in-k,in+k]~|~i\in\ZZ\}.
$
\begin{itemize}
\item If $[\frac{n-2k}{k}]=2m$, then let
\[
\mathcal{V}_1^{(n)}=\{[in+2jk,in+(2j+1)k]~|~i\in\ZZ,j=1,2,...,m-1\},
\]
\[
\mathcal{V}_2^{(n)}=\{[in+(2j-1)k,in+2jk]~|~i\in\ZZ,j=1,2,...,m-1\}\cup\{[in+2mk-k,in+2mk+2k]~|~i\in\ZZ\}.
\]
\item If $[\frac{n-2k}{k}]=2m+1$, then let
\[
\mathcal{V}_1^{(n)}=\{[in+2jk,in+(2j+1)k]~|~i\in\ZZ,j=1,2,...,m\},
\]
\[
\mathcal{V}_2^{(n)}=\{[in+(2j-1)k,in+2jk]~|~i\in\ZZ,j=1,2,...,m\}\cup\{[in+2mk+k,in+2mk+3k]~|~i\in\ZZ\}.
\]
\end{itemize}
Note that $\mathcal{V}_0^{(n)}\cup\mathcal{V}_1^{(n)}$,$\mathcal{V}_2^{(n)}$ are $k$-disjoint, $3k$-bounded and  $\mathcal{V}_0^{(n)}\cup\mathcal{V}_1^{(n)}\cup \mathcal{V}_2^{(n)}$ covers $\RR$.
Let
\[
\mathcal{W}_0^{(n)}=\{(\prod_{i=1}^nV_i)\cap X_{n}~|~V_i\in \mathcal{V}_0^{(n)}\}
\]
\[
\mathcal{W}_1^{(n)}=\{\prod_{i=1}^{j-1}\{nn_{i}\}\times V_j\times\prod_{i=j+1}^n \{nn_{i}\}~|~n_{i}\in\ZZ, V_j\in \mathcal{V}_1^{(n)},j=2,3,\cdots, n\}\cup
\{ V_1\times\prod_{i=2}^n \{nn_{i}\}~|~n_{i}\in\ZZ, V_1\in \mathcal{V}_1^{(n)}\}
\]
\[
\mathcal{W}_2^{(n)}=\{\prod_{i=1}^{j-1}\{nn_{i}\}\times V_j\times\prod_{i=j+1}^n \{nn_{i}\}~|~ V_j\in \mathcal{V}_2^{(n)},j=2,3,\cdots, n\}\cup
\{ V_1\times\prod_{i=2}^n \{nn_{i}\}~|~n_{i}\in\ZZ, V_1\in \mathcal{V}_2^{(n)}\}
\]
It is easy to see that $\mathcal{W}_i^{(n)}$ are $k$-disjoint and $3k$-bounded for $i=0,1,2$.
Let
\[\mathcal{U}_0^{(n)}=\mathcal{W}_0^{(n)}\bigcup\mathcal{W}_1^{(n)},~
\mathcal{U}_1^{(n)}=\mathcal{W}_2^{(n)}.\]
Since $\mathcal{V}_0^{(n)}\cup\mathcal{V}_1^{(n)}$ is $k$-disjoint and $3k$-bounded, $\mathcal{U}_0^{(n)}$ is $k$-disjoint and $3k$-bounded.
For every $ x\in X_{n}$, without loss of generality, we assume that $x=(x_1,...,x_{n})\in \mathbb{R}\times n\mathbb{Z}\times ... \times n\mathbb{Z}$.
Then $x$ is in one of the following cases.
\begin{itemize}
\item $x_i\in n\mathbb{Z}\text{~for~}i=2,3,\cdots$ and $x_1\in [ni-k,ni+k]$ for some $i\in \mathbb{Z}$, it is easy to see that~$x\in \bigcup \mathcal{W}_0^{(n)}$.

\item $x_i\in n\mathbb{Z}\text{~for~}i=2,3,\cdots$ and $x_1\in V$ for some $V\in \mathcal{V}_1^{(n)}$, it is easy to see that~$x\in \bigcup \mathcal{   W}_1^{(n)}$.

\item $x_i\in n\mathbb{Z}\text{~for~}i=2,3,\cdots$ and $x_1\in V$ for some $V\in \mathcal{V}_2^{(n)}$, it is easy to see that~$x\in \bigcup \mathcal{   W}_2^{(n)}$.
\end{itemize}
So $\mathcal{U}_0^{(n)}\cup \mathcal{U}_1^{(n)}$ covers $X_{n}$.
Let
\[\mathcal{U}_0=\bigcup_{n\geq 6k}\mathcal{U}_0^{(n)},\mathcal{U}_1=\bigcup_{n\geq 6k}\mathcal{U}_1^{(n)}.\]
Since $d(X_i,X_j)\geq n$ when $i,j>n$ and $i\neq j$, then $\mathcal{U}_0,\mathcal{U}_1$ are $k$-disjoint , $3k$-bounded and $\mathcal{U}_0\cup\mathcal{U}_1$ covers $\bigsqcup_{n\geq 6k}X_n$.
It follows that \[
\text{coasdim}(X_{\omega+1}\setminus \bigcup(\mathcal{U}_0\cup\mathcal{U}_1))=
\text{asdim}(X_{\omega+1}\setminus \bigsqcup_{n\geq 6k}X_n)
=\text{asdim}( \bigsqcup_{n=1}^{6k-1}X_n)
<\omega, \]
which implies coasdim$(X_{\omega+1})\leq\omega+1$. $\square$\\
\end{proof}

\begin{remark}
\begin{itemize}
\item By the Proposition \ref{prop1} and the Proposition \ref{prop2}, we can obtain that trasdim$(X_{\omega+1})=\omega+1$.
\item By the Proposition \ref{prop1} and the Lemma \ref{lem:coasdim}, we can obtain that coasdim$(X_{\omega+1})\leq\omega$ is not true.
Moreover, we obtain that coasdim$(X_{\omega+1})=\omega+1$ by the Proposition \ref{prop2}.
\end{itemize}
\end{remark}

\end{section}

\providecommand{\bysame}{\leavevmode\hbox to3em{\hrulefill}\thinspace}
\providecommand{\MR}{\relax\ifhmode\unskip\space\fi MR }
\providecommand{\MRhref}[2]{%
  \href{http://www.ams.org/mathscinet-getitem?mr=#1}{#2}
}
\providecommand{\href}[2]{#2}

\end{document}